\documentclass[12pt]{article}
\usepackage{amsmath,amsfonts,esint,amsthm,amssymb}
\usepackage{times}
\usepackage{bookmark}

\numberwithin{equation}{section}
\newtheorem{thm}{Theorem}[section]

\newtheorem{cor}[thm]{Corollary}
\newtheorem{lem}[thm]{Lemma}

\newtheorem{rmk}[thm]{Remark}

\newcommand{\Abs}[1]{\left\lvert#1\right\rvert}
\newcommand{\norm}[1]{\lVert#1\rVert}

\newcommand{\cL}{\mathcal{L}}

\newcommand{\R}{\mathbb{R}}

\newcommand{\e}{\varepsilon}
\newcommand{\ab}{{\alpha\beta}}
\newcommand{\loc}{\text{loc}}

\newcommand{\rd}{\mathbb{R}^d}
\newcommand{\varep}{\varepsilon}

\begin{document}

\title{Convergence Rates in Periodic Homogenization\\ of Systems of Elasticity}

\author{Zhongwei Shen\thanks{Department of Mathematics, University of Kentucky, Lexington, Kentucky 40506, USA.}
  \and Jinping Zhuge\thanks{Department of Mathematics, University of Kentucky, Lexington, Kentucky 40506, USA.}}

\date{}
\maketitle

\begin{abstract}
This paper is concerned with homogenization of systems of linear elasticity with rapidly oscillating 
periodic coefficients. We establish sharp convergence rates in $L^2$ for the mixed boundary value problems with bounded
measurable coefficients.
\end{abstract}

\section{Introduction and main results}
This paper is concerned with convergence rates in periodic homogenization of systems of linear elasticity with mixed boundary conditions. 
More precisely, we consider the operator
\begin{equation}\label{def_Le}
	\cL_\e = - \text{div}(A(x/\e) \nabla ) = - \frac{\partial}{\partial x_i} \left\{ a^{\alpha\beta}_{ij} \bigg( \frac{x}{\e}\bigg) \frac{\partial}{\partial x_j}\right\},  \qquad \e > 0.
\end{equation}
(The summation convention is used throughout this paper).  We will assume that the coefficient matrix $A(y) = (a^{\alpha\beta}_{ij}(y))$ with $1\le i,j, \alpha, \beta\le d$ is real, bounded measurable, and satisfies the elasticity condition,
\begin{eqnarray}\label{def_elasticity}
\begin{aligned}
	 a^{\ab}_{ij}(y)  & = a^{\beta\alpha}_{ji}(y) = a^{i\beta}_{\alpha j}(y), \\
	 \kappa_1|\xi +\xi^T |^2  & \le a^{\alpha\beta}_{ij}(y) \xi_i^\alpha \xi_j^\beta \le \kappa_2 |\xi|^2,
\end{aligned}
\end{eqnarray}
for  $y\in\R^d $ and matrix $ \xi = (\xi_i^\alpha) \in \R^{d\times d}$, where $\kappa_1, \kappa_2 > 0$. We also assume that $A$ satisfies the 1-periodic condition:
\begin{equation}\label{def_periodicity}
A(y+z) = A(y) \qquad \text{ for  }\ y\in \R^d \text{ and } z\in \mathbb{Z}^d.
\end{equation}

We shall be interested in the mixed boundary value problems (or mixed problems) 
for the elliptic system $\mathcal{L}_\varep (u_\varep)=F$ in a bounded Lipschitz domain $\Omega$.
Let $D$ be a closed subset of $\partial\Omega$ 
and $N = \partial \Omega \setminus D$. 
Denote by $H^1_D(\Omega; \rd)$ the closure in $H^1(\Omega;\rd)$ of the set $C_0^\infty (\R^d\setminus D; \rd)$ and $H^{-1}_D (\Omega ;\rd)$ the dual of $H^1_D(\Omega;\rd)$.
 Assume that $F\in H^{-1}_D(\Omega;\rd), f\in H^1(\Omega;\rd)$ and $g\in H^{-1/2}(\partial\Omega;\rd)$ (the dual of 
 $H^{1/2}(\partial\Omega;\rd)$). We call $u\in H^1(\Omega;\rd)$  a weak solution of the mixed boundary value problem
\begin{align}
\left\{
\begin{aligned}\label{eq_mbvp}
\mathcal{L}_\varepsilon (u_\varepsilon) &= F &\quad& \text{ in } \Omega, \\
u_\varepsilon &= f &\quad&\text{ on } D, \\
n\cdot A(x/\varepsilon) \nabla u_\varepsilon &= g &\quad &\text{ on } N,
\end{aligned}
\right.
\end{align}
if $u_\varep-f \in H^1_D(\Omega;\rd)$ and 
\begin{equation}\label{eq_int_id}
\int_{\Omega} A^\varep  \nabla u_\varep \cdot \nabla \varphi = \langle F, \varphi\rangle_{H^{-1}_D(\Omega)
\times H^1_D(\Omega)}
 + \langle g, \varphi\rangle_{H^{-1/2}(\partial\Omega)\times H^{1/2}(\partial\Omega)}
 \end{equation}
 holds for any $\varphi \in H^1_D(\Omega; \rd)$.
 Here and throughout this paper, 
 we define $h^\e(x) = h(x/\e)$ for any function $h$ and use $n$ to denote the outward unit normal to 
 $\partial\Omega$.
 

The existence and uniqueness of the weak solution to the mixed problem (\ref{eq_mbvp})
follow readily from the Lax-Milgram theorem, with the help of Korn's inequalities.
It can also be shown that under the elasticity condition (\ref{def_elasticity}) and the periodicity condition (\ref{def_periodicity}), the weak solutions $u_\e$ converge to some function $u_0$ weakly in $H^1(\Omega;\rd) $ and thus strongly in $L^2(\Omega;\rd)$,
as $\e\to 0$. 
Furthermore, the function $u_0$ is the weak solution to the mixed problem:
\begin{align}
\left\{
\begin{aligned}\label{eq_u0_0}
\cL_0 u_0 &= F&\quad& \text{ in } \Omega, \\
u_0 &= f &\quad& \text{ on } D, \\
n \cdot \widehat{A} \nabla u_0 &= g & \quad & \text{ on } N,
\end{aligned}
\right.
\end{align}
where
\begin{equation}\label{def_homo_L0}
\cL_0 = -\text{div}(\widehat{A} \nabla) = \frac{\partial}{\partial x_i} \left\{ \widehat{a}^{\alpha\beta}_{ij} \frac{\partial}{\partial x_j}\right\}
\end{equation}
is a system of linear elasticity with constant matrix
$\widehat{A} = (\widehat{a}^{\alpha\beta}_{ij})$, known as the homogenized (or effective) matrix of $A$.

The primary purpose of this paper is to establish the optimal rate of convergence 
of $u_\varep$ to $u_0$ in $L^2(\Omega;\rd)$. More precisely, we are interested in the estimate,
\begin{equation}\label{optimal}
\| u_\varep -u_0\|_{L^2(\Omega)} \le C\, \varep\, \| u_0\|_{H^2(\Omega)},
\end{equation}
for the mixed problem (\ref{eq_mbvp}) with nonsmooth coefficients,
where $C$ depends at most on $d$, $\kappa_1$, $\kappa_2$, $\Omega$, and $D$.
The problem of convergence rates  is central in quantitative homogenization and has been studied extensively
in various settings. We refer the reader to \cite{b-1978, JKO,OSY} for references on earlier work in this area.
More recent work on the problem of convergence rates in periodic homogenization 
may be found in \cite{ZP, Griso-2004, GG, Pas-2006, Onofrei-2007, KLS1, KLS2, PS, Su1, Su2,  Sh,  Gu} and their references.
In particular,  the estimate (\ref{optimal}) was proved
by Griso in \cite{Griso-2004, GG} for scalar elliptic equations with either Dirichlet or Neumann 
boundary conditions, using the method of periodic unfolding \cite{CDG,CDG1}.
In \cite{Su1, Su2} the results were extended by Suslina to a broader class of elliptic systems in $C^2$ domains, which
includes the systems of elasticity considered in this paper, with either Dirichlet or Neumann boundary
conditions.
We mention that for systems of elasticity, the results were further extended by the  first author in \cite{Sh}, where
the estimate $\| u_\e -u_0\|_{L^p(\Omega)}\le C\, \e\, \| u_0\|_{L^2(\Omega)}$,
with $p=\frac{2d}{d-1}$, was proved  in Lipschitz domains for solutions with either Dirichlet or Neumann boundary 
conditions. 
As far as we know, there are no results on the estimate (\ref{optimal}) for the mixed boundary value problems,
even for scalar elliptic equations.

The following is our main result.

\begin{thm}\label{thm_L2}
	Let $\Omega$ be a bounded $C^{1,1}$ domain and $D$  a closed subset of $\partial\Omega$ with a nonempty interior. Let $u_\e, u_0$ be the weak solutions of mixed boundary value problems (\ref{eq_mbvp}) and (\ref{eq_u0_0}), respectively. Assume that $u_0\in H^2(\Omega; \rd)$. Then the estimate (\ref{optimal})	holds with constant $C$ depending at most on $d$, $\kappa_1$, $\kappa_2$,
	$D$, and $\Omega$.
\end{thm}

Let $\chi=(\chi_j^{\alpha\beta})$ denote  the correctors for the operator $\mathcal{L}_\e$.
Let $S_\e$ be a smoothing operator at $\e$-scale and $\widetilde{u}_0$ an extension of $u_0$ 
from $H^2(\Omega;\rd)$ to $H^2(\R^d;\rd)$. The key step in the proof of Theorem \ref{thm_L2}
 is the following estimate, 
\begin{align}
\begin{aligned}\label{ineq_main}
&\Abs{\int_{\Omega} A^\e \nabla\Big(u_\e - u_0 - \e \chi^\e S_\e (\nabla\widetilde{u}_0)\Big) \cdot \nabla \psi} \\
&\qquad \le C\Big\{ \e\,   \norm{\nabla \psi}_{L^2(\Omega)} + \e^{1/2} \norm{\nabla \psi}_{L^2(\Omega_{2\e})}
\Big\} \norm{u_0}_{H^2(\Omega)},
\end{aligned}
\end{align}
where $\psi\in H^1_D(\Omega; \rd)$ and $\Omega_{2\e} = \{x\in \Omega:\text{dist}(x,\partial\Omega) < 2\e \}$
(see Lemma \ref{lem_BeSedu0}).
We point out that some analogous estimates were proved in \cite{GG} by the method of periodic unfolding, 
which is not used in this paper. Our approach to (\ref{ineq_main}), which involves a standard smoothing operator 
at the scale $\e$, is much more direct and flexible and allows us to handle different boundary 
conditions in a uniform fashion.
We also mention that the use of smoothing operators as well as the duality argument in our proof of
Theorem \ref{thm_L2} is motivated by the work \cite{GG,Su1,Su2}.
However, in comparison with \cite{Su1,Su2}, our proof does not rely on the sharp convergence estimates
for the whole space $\rd$ and thus avoids the estimates of terms that are used to correct the boundary
discrepancies. As a result, this significantly simplifies the argument.

As a bi-product, we also obtain an $O(\varep^{1/2})$ estimate in $H^1(\Omega)$ as well as an interior 
$O(\varep)$ estimate in $H^1$.

\begin{thm}\label{thm_H1}
	Under the same conditions as in Theorem \ref{thm_L2}, we have
	\begin{equation}
		\norm{u_\e -  u_0 - \e \chi^\e S_\e (\nabla \widetilde{u}_0)}_{H^1(\Omega)} \le C\, \e^{1/2} \norm{u_0}_{H^2(\Omega)},
	\end{equation}
	where $C$ depends at most on $d$, $\kappa_1$, $\kappa_2$, $D$, and $\Omega$.
\end{thm}

\begin{thm}\label{thm_H1_int}
	Under the same condition as Theorem \ref{thm_L2}, we have
	\begin{equation}
	\norm{\delta \nabla (u_\e -  u_0 - \e \chi^\e S_\e (\nabla \widetilde{u}_0)) }_{L^2(\Omega)} \le C\,\e\, \norm{u_0}_{H^2(\Omega)},
	\end{equation}
	where $\delta(x) =\text{dist}(x, \partial\Omega)$ and
	$C$ depends at most on $d$, $\kappa_1$, $\kappa_2$, $D$, and $\Omega$.
\end{thm}

 We mention  that our argument  also yields the estimates in Theorems \ref{thm_L2}, \ref{thm_H1} and
 \ref{thm_H1_int} for the Neumann problem,
 where $D=\emptyset$. We further point out that the approach
 works equally well for the strongly  elliptic systems
 $-\text{div} (A(x/\varep)\nabla u_\varep)=F$, where $A(y)= (a_{ij}^{\alpha\beta} (y))$
 with $1\le i,j\le d$ and $1\le \alpha, \beta\le m$ is real, bounded measurable, 1-periodic, and
 satisfies the ellipticity condition $a_{ij}^{\alpha\beta} (y)\xi_i^\alpha\xi_j^\beta\ge \mu |\xi|^2$
 for $y\in \rd$ and $\xi =(\xi_i^\alpha) \in\mathbb{R}^{m\times d}$.
 

\section{Preliminaries}

In this section we give a brief review of the solvability and the homogenization theory for the mixed problem (\ref{eq_mbvp}). We begin with a Korn inequality.

\begin{lem}\label{lem_Korn}
	Let $\Omega$ be a bounded Lipschitz domain in $\rd$
	 and $D$ a closed subset of $\partial\Omega$ with a nonempty interior. 
	 Then for any vector field $u  \in H^1_D(\Omega;\rd)$,
	\begin{equation}\label{ineq_Korn1}
	\norm{u}_{H^1(\Omega)} \le C\,  \norm{\nabla u+(\nabla u)^T}_{L^2(\Omega)},
	\end{equation}
	where  $C$ depends only on $d,D$, and $\Omega$.
\end{lem}

\begin{proof}
Since $D$ has a nonempty  interior in $\partial\Omega$, there exist $x_0\in \partial\Omega$ and $r_0>0$
such that $B(x_0, r_0)\cap \partial\Omega\subset D \subset \partial\Omega$.
As a result, the inequality (\ref{ineq_Korn1}) follows from  \cite[Theorem 2.7]{OSY}. 
\end{proof}

\begin{thm}\label{thm_exist}
	Let $\Omega$ be a bounded Lipschitz domain in $\rd$ and $D$ a closed subset of $\partial\Omega$ with a nonempty interior.
	For $F\in H^{-1}_D (\Omega; \rd)$, $f\in H^1(\Omega; \rd)$ and $g\in H^{-1/2}(\partial\Omega; \rd)$,
	  there exists a unique weak solution $u_\varep \in H^1(\Omega; \rd)$
	to the mixed problem (\ref{eq_mbvp}). Moreover, the solution $u_\varep$ satisfies 
	\begin{equation}
	\norm{u_\e}_{H^1(\Omega)} \le C\Big\{ \norm{F}_{H^{-1}_D (\Omega)}
	 + \norm{f}_{H^1(\Omega)} + \norm{g}_{H^{-1/2}(\partial\Omega)}\Big\},
	\end{equation}
	where $C$ depends only on $d$, $\kappa_1$, $\kappa_2$,  $\Omega$, and $D$.
\end{thm}

\begin{proof} By considering the bilinear form 
$$
\int_\Omega A^\e \nabla \psi\cdot \nabla \varphi
$$
and the bounded linear functional
$$
\langle F, \varphi\rangle_{H^{-1}_D(\Omega)\times H^1_D(\Omega)}
+\langle g, \varphi\rangle_{H^{-1/2}(\partial\Omega) \times H^{/2}(\partial\Omega)}
-\int_\Omega A^\e \nabla f \cdot \nabla \varphi
$$
on $H^1_D(\Omega; \rd)$, Theorem \ref{thm_exist}  follows readily from
the Lax-Milgram theorem, using the elasticity condition (\ref{def_elasticity}) and the Korn inequality in Lemma \ref{lem_Korn}. 
\end{proof}

Assume that $A$ satisfies (\ref{def_elasticity}) and (\ref{def_periodicity}). Let $\chi = (\chi_j^\beta) = (\chi_j^\ab)$ denote the correctors for $\cL_\e$, where $1\le j\le d$ and $1\le \alpha,\beta\le d$. 
This means that  $\chi_j^\beta \in H^1_{\loc}(\rd; \rd)$ is the  1-periodic function such that $\int_Q \chi_j^\beta=0$ and
\begin{equation}\label{homo-eq}
\cL_1 (\chi_j^\beta + P_j^\beta) = 0 \quad \text{in } \R^d,
\end{equation}
where $Q = [-1/2,1/2]^d$, $P_j^\beta(y) = y_je^\beta $, and $e^\beta =  (0,\cdots,1,\cdots,0)\in \rd$ 
 with 1 in the $\beta$th position. 
  For the existence of correctors $\chi$, see e.g. \cite{JKO, OSY}.
The homogenized operator $\mathcal{L}_0$ is given by (\ref{def_homo_L0}), 
where the homogenized matrix $\widehat{A} = (\widehat{a}_{ij}^\ab)$ is defined by
\begin{equation}\label{homo-coeff}
\widehat{A} = \fint_Q A(I+\nabla \chi) \quad \text{or precisely} 
\quad \widehat{a}_{ij}^\ab = \fint_Q \left\{ a_{ij}^\ab + a_{ik}^{\alpha\gamma} \frac{\partial}{\partial y_k} (\chi_j^{\gamma\beta}) \right\}.
\end{equation}
 It is known that $\widehat{A}$ satisfies the elasticity condition (\ref{def_elasticity})
(with possible different $\kappa_1, \kappa_2$) \cite{JKO}. 

\begin{thm}\label{thm_homo}
	Let $\Omega$ be a bounded Lipschitz domain in $\rd$ 
	and $D$ a closed subset of $\partial\Omega$ with a nonempty interior. 
	For $\varep>0$, let $u_\e, u_0$ be the weak solutions of the mixed boundary value problems (\ref{eq_mbvp}) and (\ref{eq_u0_0}), respectively, 
	where $F \in H^{-1}_D(\Omega; \rd)$, $f\in H^1(\Omega; \rd)$, and $g\in H^{-1/2}(\partial\Omega; \rd)$. Then
	\begin{equation}
	\aligned
	u_\e  &\rightharpoonup u_0 & \  &\text{ weakly in } \ H^1(\Omega;\rd),\\
	A^\e \nabla u_\e  & \rightharpoonup \widehat{A}\nabla u_0 & \  &
	\text{ weakly in }\ L^2(\Omega; \mathbb{R}^{d\times d}),
	\endaligned
	\end{equation}
	as $\e\to 0$.
\end{thm}

\begin{proof} The proof is the same as in the case of the Dirichlet problem \cite{JKO}.
 By Theorem \ref{thm_exist} the solutions $u_\e$ are uniformly bounded in $ H^1(\Omega;\rd)$.
Let $\{ u_{\e^\prime}\} $ be a subsequence  such that
$$
\aligned
u_{\e^\prime} & \rightharpoonup w \ \text{ weakly in } \ H^1(\Omega; \rd),\\
A^{\e^\prime}\nabla u_{\e^\prime}  &\rightharpoonup G \ \text{ weakly in } \ L^2(\Omega; \mathbb{R}^{d\times d}).
\endaligned
$$
Since $u_\varep -f \in H^1_D(\Omega; \rd)$,
we have $w- f\in H^1_D(\Omega; \rd)$. 
Next we will show that $G = \widehat{A} \nabla w$. 
To this end  we consider identity
	\begin{equation}\label{id_homo}
	\int_\Omega A^{\e^\prime} \nabla u_{\e^\prime} \cdot \nabla \Big(P_j^\beta + \e^\prime \chi^{\beta}_j(x/\e^\prime)\Big) \phi
	=\int_\Omega  \nabla u_{\e^\prime} \cdot A^{\e^\prime} \nabla \Big(P_j^\beta + \e^\prime \chi^{\beta}_j(x/\e^\prime)\Big) \phi,
	\end{equation}
	where $\phi\in C_0^\infty(\Omega)$ and
	we have used the symmetry condition $a_{ij}^{\alpha\beta}=a_{ji}^{\beta\alpha}$.
	By the Div-Curl lemma (see e.g. \cite[p.4]{JKO}), the LHS of (\ref{id_homo}) converges to
	\begin{equation}
	\int_\Omega G \cdot \Big(\nabla P_j^\beta \Big) \phi = \int_\Omega G_j^\beta \phi,
	\end{equation}
	as $\e \to 0$, where $G=(G_i^\alpha)$. 
	Similarly, by the Div-Curl lemma, the RHS of (\ref{id_homo}) converges to
	\begin{equation}
	\int_\Omega \nabla w \cdot \left( \fint_Q A \big(\nabla P_j^\beta + \nabla \chi^{\beta}_j\big)\right) \phi
	  = \int_\Omega \frac{\partial w^\alpha}{\partial x_i}\cdot \widehat{a}^{\alpha\beta}_{ij} \phi,
	\end{equation}
	as $\e \to 0$. 
	Since $\phi \in C_0^\infty(\Omega)$ is arbitrary,
	we obtain 
	$$
	G_j^\beta =\frac{\partial w^\alpha}{\partial x_i} \widehat{a}^{\alpha\beta}_{ij}=\widehat{a}_{ji}^{\beta\alpha} 
	\frac{\partial w^\alpha}{\partial x_i};
	$$
	i.e. $G=\widehat{A}\nabla w$ in $\Omega$.
	
	Finally, note that for any $\varphi \in H^1_D(\Omega; \rd)$,
	$$
	\aligned
	\int_\Omega \widehat{A}\nabla w \cdot \nabla \varphi & =\int_\Omega G\cdot \nabla \varphi
	=\lim_{\e^\prime \to 0} \int_\Omega A^{\e^\prime} \nabla u_{\e^\prime}\cdot \nabla \varphi\\
	&=\langle F, \varphi\rangle_{H^{-1}_D(\Omega)\times H^1_D(\Omega)}
	+\langle g, \varphi\rangle_{H^{-1/2}(\partial\Omega)\times  H^{1/2}(\partial\Omega)}.
	\endaligned
	$$
	This shows that $w$ is a solution of the mixed problem (\ref{eq_u0_0}) for the homogenized system.
	By the uniqueness of  (\ref{eq_u0_0}) it follows that the whole sequence $u_\e$ converges 
	weakly to $u_0$ in $H^1(\Omega; \rd)$. The argument above also shows that the whole sequence $A^\varep\nabla u_\e$
	converges weakly to $\widehat{A}\nabla u_0$ in $L^2(\Omega; \mathbb{R}^{d\times d})$.
	\end{proof}


\section{Convergence rates in $H^1(\Omega)$}
 
 In this section we give the proof of the estimate (\ref{ineq_main}) and Theorem \ref{thm_H1}.
Let $S_\e$ be the operator on $L^2(\R^d)$ given by
\begin{equation}
S_\e u(x) = u*\phi_\e(x) = \int_{\R^d} u(x-y)\phi_\e (y)dy,
\end{equation}
where $\phi_\e(x) = \e^{-d}\phi(\e^{-1}x)$, $\phi\in C_0^{\infty}(B(0,1/2))$, $\phi\ge 0$,
 and $\int \phi = 1$. We will call $S_\e$ the smoothing operator at  $\e$-scale. Note that
\begin{equation}
\norm{S_\e u}_{L^2(\R^d)} \le \norm{u}_{L^2(\R^d)},
\end{equation}
and  $D^\alpha S_\e u = S_\e D^\alpha u$ for $u\in H^s(\R^d)$ and $|\alpha| \le s$.

\begin{lem}\label{lem_Se}
	Let $u\in H^1(\R^d)$. Then
	\begin{equation}
	\norm{S_\e u - u}_{L^2(\R^d)} \le C\, \e\, \norm{\nabla u}_{L^2(\R^d)},
	\end{equation}
	for any $\e > 0$.
\end{lem}

\begin{proof}
	This is well known. See e.g. \cite{ZP} or \cite{Sh} for a proof.
\end{proof}

\begin{lem}\label{lem_feSe}
	Let $f\in L^2_\loc(\rd) $ be a 1-periodic function. Then for any $u\in L^2(\R^d)$,
	\begin{equation}\label{ineq_feseu}
	\norm{f^\e S_\e u}_{L^2(\R^d)} \le C\,\norm{f}_{L^2(Q)} \norm{u}_{L^2(\R^d)},
	\end{equation}
	where $f^\e (x)=f(x/\e)$ and $Q=[-1/2, 1/2]^d$.
\end{lem}

\begin{proof}
See e.g. \cite{ZP} or \cite{Sh} for a proof.
	\end{proof}

Let  $\widetilde{\Omega}_{\e} = \big\{x\in \rd: \text{dist}(x,\partial\Omega) < \e\big\}$.

\begin{lem}\label{lem_Oeu}
	Let $\Omega$ be a bounded Lipschitz domain in $\rd$. Then  for any $u\in H^1(\R^d)$,
	\begin{equation}
	\int_{\widetilde{\Omega}_{\e}} |u|^2 \le C \,\e\, \norm{u}_{H^1(\R^d)} \norm{u}_{L^2(\R^d)},
	\end{equation}
	where the constant $C$ depends only on  $d$ and $\Omega$.	
\end{lem}
\begin{proof} This is known. See e.g. \cite{PS}.  We provide a proof for the reader's convenience.
	Note that the desired estimate is invariant under Lipschitz homeomorphism. 
	By covering $\partial\Omega$ with coordinate patches,
	 it suffices to prove a local estimate for the upper half-space with $0<\e<1$. 
	 	 
	Let $\theta\in C^\infty(\mathbb{R})$ such that $0\le \theta \le 1$, $\theta (t)=1$ for $t\le 1$, and $\theta (t)=0$ for $t\ge 2$.
	 For any $(x',t)$ with $x^\prime \in \mathbb{R}^{d-1}$ and $-\varep<t<\e<1$, we have
	\begin{align*}
	u^2(x',t) &= -\int_t^2 \frac{\partial}{\partial s}  \big[\theta (s) u^2(x',s)\big] \, ds \\
	& = -\int_t^2 \frac{\partial}{\partial s}   \big[ \theta(s) \big] u^2(x',s) \, ds - 2\int_t^2 \theta(s) u (x^\prime, s)
	 \frac{\partial}{\partial s}   u(x',s) \, ds.
	\end{align*}
	It follows that
	\begin{equation}
	u^2(x',t) \le C \int_{-2}^2 u^2(x',s)\, ds + 2\int_{-2}^2 |u(x',s)||\nabla u(x',s)| \, ds.
	\end{equation}
	Let $\Delta$ be a surface ball in $\mathbb{R}^{d-1}$. Then 
	\begin{align*}
	&\int_{-\e}^\e \int_\Delta u^2(x',t)\, dx'dt \\
	&\qquad \le C \e \int_{-2}^2\int_\Delta u^2(x',s)\, dx'ds + 4\e\int_{-2}^2\int_\Delta |u(x',s)||\nabla u(x',s)| \, dx'ds \\
	& \qquad \le C\, \e\,  \norm{u}_{L^2(\Delta\times [-2,2])} \norm{u}_{H^1(\Delta\times[-2,2])}.
	\end{align*}
	This completes the proof.
\end{proof}

\begin{lem}\label{lem_OefeSe}
	Let $\Omega$ be a bounded Lipschitz domain in $\rd$
	and $f\in L^2_{\loc}(\rd)$  a 1-periodic function.
	Then  for any $u\in H^1(\R^d)$,
	\begin{equation}
	\int_{\widetilde{\Omega}_{\e}} |f^\e|^2 |S_\e u|^2 \le C\, \e\, \norm{f}_{L^2(Q)}^2 \norm{u}_{H^1(\R^d)} \norm{u}_{L^2(\R^d)},
	\end{equation}
	where  $C$ depends only on $d$ and  $\Omega$.
\end{lem}

\begin{proof} This is known and similar estimates may be found in \cite{ZP, PS}.
	Note that
	\begin{equation}
	S_\e u(x) = \int_{B(0,1/2)} u(x-\e y) \phi(y) \, dy.
	\end{equation}
	By Minkowski's integral inequality and Fubini's theorem,
	\begin{align*}
	\int_{\widetilde{\Omega}_\e} |f^\e(x)|^2 |S_\e u(x)|^2\, dx 
	&\le C \int_{\widetilde{\Omega}_\e} \int_{B(0,1/2))} |f^\e(x)|^2 |u(x-\e y)|^2\, dydx \\
	& \le C \int_{B(0,1/2)} \int_{\widetilde{\Omega}_\e - \e y} |f^\e(x+\e y)|^2 |u(x)|^2 \,dx dy \\
	& \le C \int_{B(0,1/2))} \int_{\widetilde{\Omega}_{2\e}} |f^\e(x+\e y)|^2 |u(x)|^2 \, dx dy \\
	& \le C \int_{{\widetilde{\Omega}}_{2\e}} |u(x)|^2 \, dx \sup_{x\in \R^d} \int_{B(0,1/2)}  |f^\e(x+\e y)|^2 \, dy \\
	& \le C\,\e\, \norm{f}_{L^2(Q)}^2 \norm{u}_{H^1(\R^d)} \norm{u}_{L^2(\R^d)},
	\end{align*}
	where we have used Lemma \ref{lem_Oeu} for the last inequality.
\end{proof}

Let $u_0$ be the solution of (\ref{eq_u0_0}).
Suppose that $u_0\in H^2(\Omega; \rd)$.
Since $\Omega$ is Lipschitz, there exists a bounded extension operator $E:H^2(\Omega;\rd) 
\to H^2(\R^d;\rd)$ so that $\widetilde{u}_0 = E u_0$ is an extension of $u_0$
and  $\norm{\widetilde{u}_0}_{H^2(\rd)} \le C\norm{u_0}_{H^2(\Omega)}$. Let
\begin{equation}\label{def_we}
w_\e = u_\e -  u_0 - \e \chi^\e S_\e \nabla \widetilde{u}_0,
\end{equation}
where $u_\e \in H^1(\Omega; \rd)$ is the solution of (\ref{eq_mbvp}).
Then $w_\e$ satisfies
\begin{align}
\left\{
\begin{aligned}\label{eq_we_11}
\cL_\e w_\e &= F_\e = \cL_0 u_0 - \cL_\e u_0 - \cL_\e (\e \chi^\e S_\e \nabla \widetilde{u}_0) &\quad &\text{in } \Omega, \\
w_\e &= h_\e = -\e \chi^\e S_\e \nabla \widetilde{u}_0 &\quad &\text{on } D, \\
n\cdot A^\e \nabla w_\e  &= g_\e = n\cdot \widehat{A}\nabla u_0 - n\cdot A^\e \nabla u_0
 - n\cdot A^\e \nabla (\e \chi^\e S_\e \nabla \widetilde{u}_0) &\quad &\text{on } N.
\end{aligned}
\right.
\end{align}

Recall that $\Omega_{2\e} =\big\{ x\in \Omega:\, \text{dist}(x, \partial\Omega)<2\e\big\}$.
The following lemma plays a key role in this paper.

\begin{lem}\label{lem_BeSedu0}
Let $\Omega$ be a bounded Lipschitz domain in $\rd$ and $D$ a closed subset of $\partial\Omega$.
	For any $\psi\in H^1_D(\Omega; \rd)$, we have
	\begin{align*}
		\Abs{\int_{\Omega} A^\e \nabla w_\e \cdot \nabla \psi} \le C \, \norm{u_0}_{H^2(\Omega)} 
		\Big\{ \e\,  \norm{\nabla \psi}_{L^2(\Omega)} +  \e^{1/2} \norm{\nabla \psi}_{L^2(\Omega_{2\e})}\Big\},
	\end{align*}
	where $w_\e$ is given by (\ref{def_we}) and
	$C$ depends only on $d$, $\kappa_1$, $\kappa_2$, $D$, and $\Omega$.
\end{lem}

\begin{proof} By a density argument we may assume  $\psi\in C_0^\infty(\rd \setminus D; \rd)$.
Using
$$
\int_\Omega A^\e \nabla u_\e \cdot \nabla\psi =\int_\Omega \widehat{A}\nabla u_0 \cdot \nabla \psi,
$$
we obtain
\begin{equation}\label{eq_Aedwe dpsie}
\int_{\Omega} A^\e \nabla w_\e \cdot \nabla \psi = \int_{\Omega} 
\Big[\widehat{A}\nabla u_0 - A^\e \nabla u_0 - \e A^\e\nabla(\chi^\e S_\e \nabla\widetilde{u}_0)\Big] \cdot \nabla \psi.
\end{equation}
A direct calculation shows that
\begin{align*}
&\widehat{A}\nabla u_0 - A^\e \nabla u_0 - \e A^\e\nabla(\chi^\e S_\e \nabla\widetilde{u}_0)\\
& \qquad = 
B^\e S_\e \nabla \widetilde{u}_0 + \left[(\widehat{A} \nabla u_0 
- \widehat{A} S_\e \nabla \widetilde{u}_0) - (A^\e \nabla u_0 
- A^\e S_\e \nabla \widetilde{u}_0) - \e A^\e \chi^\e S_\e\nabla^2 \widetilde{u}_0 \right] \\
& \qquad = B^\e S_\e \nabla \widetilde{u}_0 + T_\e,
\end{align*}
where $B(y) =\widehat{A} - A(y) - A(y)\nabla \chi(y)$. As a result, we have
\begin{align}
\begin{aligned}\label{eq_J1J2}
\int_{\Omega} A^\e \nabla w_\e \cdot \nabla \psi
 &= \int_{\Omega} B^\e S_\e \nabla \widetilde{u}_0 \cdot \nabla \psi + \int_{\Omega} T_\e \cdot \nabla \psi \\
&= J_1 + J_2.
\end{aligned}
\end{align}
For $J_2$, it follows from Lemmas \ref{lem_Se} and \ref{lem_feSe} that
\begin{equation}
\norm{T_\e}_{L^2(\Omega)} \le C\, \e \, \norm{u_0}_{H^2(\Omega)}.
\end{equation}
Thus,
\begin{equation}\label{ineq_J2}
|J_2| \le C\,\e\, \norm{u_0}_{H^2(\Omega)}\norm{\nabla \psi}_{L^2(\Omega)}.
\end{equation}

To handle $J_1$, we write 
\begin{equation}
\aligned
J_1  & = \int_{\Omega} B^\e (1-\theta_\e)S_\e\nabla \widetilde{u}_0 \cdot \nabla \psi 
+ \int_{\Omega} B^\e \theta_\e S_\e\nabla \widetilde{u}_0 \cdot \nabla \psi \\
&= J_{11} + J_{12},
\endaligned
\end{equation}
where $\theta_\e\in C_0^\infty(\rd)$ is a smooth function
 such that $\theta_\e(x) = 1$ if $x\in \widetilde{\Omega}_\e$, $\theta_\e(x) = 0$ 
if $x\notin \widetilde{\Omega}_{2\e}$, and $|\nabla \theta_\e| \le C\e^{-1}$.
Since $B(y)$ is 1-periodic and locally square integrable, by Lemma \ref{lem_OefeSe}, we obtain
\begin{equation}\label{J-12}
\aligned
|J_{12}|  & \le \int_{\Omega_{2\e}} |B^\e S_\e\nabla \tilde{u}_0 \cdot \theta_\e\nabla \psi | \\
&\le C\e^{1/2} \norm{u_0}_{H^2(\Omega)} \norm{\nabla \psi}_{L^2(\Omega_{2\e})}.
\endaligned
\end{equation}

It remains to estimate $J_{11}$.
To this end we let $B=(b_{ij}^{\alpha\beta} (y))$. Note that $b_{ij}^{\alpha\beta}$ is 1-periodic and
$b_{ij}^{\alpha\beta}\in L^2_{\loc}(\rd)$. Also, by (\ref{homo-eq}) and (\ref{homo-coeff}),
$$
\frac{\partial}{\partial y_i} b_{ij}^{\alpha\beta}=0 \quad \text{ and } \quad \int_Q b_{ij}^{\alpha\beta}=0.
$$
It follows that there exist  1-periodic functions $\phi_{kij}^{\alpha\beta}\in
H^1_{\loc}(\rd)$, where $1\le \alpha, \beta, i, j, k\le d$,  such that
	\begin{equation}\label{eq_By}
	b_{ij}^{\alpha\beta}= \frac{\partial}{\partial y_k} \phi_{kij}^{\alpha\beta}
	 \qquad \text{and} \qquad \phi_{kij}^{\alpha\beta} = - \phi_{ikj}^{\alpha\beta}.
	\end{equation}
(see \cite{JKO} or \cite{KLS1}).
Using integration by parts, this allows us to write $J_{11}$ as
\begin{align*}
J_{11} & = \int_{\Omega} \frac{\partial }{\partial x_k}\left(\e\phi^{\ab\e}_{kij}\right)
 (1-\theta_\e) S_\e \left(\frac{\partial \widetilde{u}_0^\beta}{\partial x_j}\right)
 \cdot  \frac{\partial \psi^\alpha}{\partial x_i} \\
& = - \e \int_{\Omega} \phi^{\ab\e}_{kij} \frac{\partial }{\partial x_k}(1-\theta_\e) 
S_\e \left(\frac{\partial \widetilde{u}^\beta_0}{\partial x_j}\right)
\cdot \frac{\partial \psi^\alpha}{\partial x_i} 
 - \e \int_{\Omega} \phi^{\ab\e}_{kij} (1-\theta_\e)  S_\e\left(\frac{\partial^2 \widetilde{u}_0}{\partial x_k\partial x_j} \right)
  \cdot  \frac{\partial \psi^\alpha}{\partial x_i} \\
&\qquad - \e \int_{\Omega} \phi^{\ab\e}_{kij} (1-\theta_\e) S_\e
\left( \frac{\partial \widetilde{u}_0^\beta}{\partial x_j} \right)\cdot \frac{\partial^2 \psi^\alpha}{\partial x_i \partial x_k},
\end{align*}
where $\phi_{kij}^{\alpha\beta\e} (x) =\phi_{kij}^{\alpha\beta} (x/\e)$.
Note that the last term vanishes in view of the second equation in (\ref{eq_By}).
Therefore, by Lemmas \ref{lem_feSe} and \ref{lem_OefeSe}, we obtain
\begin{align*}
|J_{11}| &\le C \int_{\Omega_{2\e}} |\Phi^\e S_\e\nabla \widetilde{u}_0| |\nabla \psi|
 + C\,\e \int_{\Omega} |\Phi^\e S_\e\nabla^2 \widetilde{u}_0| |\nabla \psi| \\
& \le C\, \e^{1/2} \norm{u_0}_{H^2(\Omega)}\norm{\nabla \psi}_{L^2(\Omega_{2\e})} 
+ C\, \e\, \norm{u_0}_{H^2(\Omega)}\norm{\nabla \psi}_{L^2(\Omega)},
\end{align*}
where $\Phi= (\phi_{kij}^{\alpha\beta})$.
Thus, in view of (\ref{J-12}), we have proved that
\begin{equation}\label{ineq_J1}
|J_1| \le C\, \e^{1/2} \norm{u_0}_{H^2(\Omega)}\norm{\nabla \psi}_{L^2(\Omega_{2\e})} 
+ C\, \e\, \norm{u_0}_{H^2(\Omega)}\norm{\nabla \psi}_{L^2(\Omega)}.
\end{equation}
The lemma now follows by combining (\ref{eq_J1J2}), (\ref{ineq_J2}), and (\ref{ineq_J1}).
\end{proof}

We are ready to give the proof of Theorem \ref{thm_H1}.

\begin{proof}[\bf Proof of Theorem \ref{thm_H1}]

Let $w_\e $ be defined by (\ref{def_we}).
Set $r_\e = \e \theta_\e \chi^\e S_\e (\nabla \widetilde{u}_0) $ and $\psi_\e = w_\e + r_\e$,
where $\theta_\e\in C_0^\infty(\rd)$ is the same as in the proof of Lemma \ref{lem_OefeSe}.
Then 
$$
\psi_\e =u_\e -u_0 -\e (1-\theta_\e)\chi^\e S_\e (\nabla \widetilde{u}_0)\in H^1_D(\Omega;\rd).
$$ 
It follows from Lemma \ref{lem_BeSedu0} that
\begin{equation}\label{ineq_Ae we psie}
\Abs{\int_{\Omega} A^\e \nabla w_\e \cdot \nabla \psi_\e} \le C\,\e^{1/2} \norm{u_0}_{H^2(\Omega)} \norm{\nabla \psi_\e}_{L^2(\Omega)}.
\end{equation}
This, together with the observation  $w_\e =\psi_\e- r_\e$ and
\begin{equation}\label{ineq_re}
\norm{r_\e}_{H^1(\Omega)} \le C\, \e^{1/2} \norm{u_0}_{H^2(\Omega)},
\end{equation}
gives
\begin{equation}\label{estimate-H-1}
\Abs{\int_{\Omega} A^\e \nabla \psi_\e \cdot \nabla \psi_\e}
 \le C\e^{1/2} \norm{u_0}_{H^2(\Omega)} \norm{\nabla \psi_\e}_{L^2(\Omega)}.
\end{equation}
By the Korn inequality (\ref{ineq_Korn1}), the elasticity condition (\ref{def_elasticity}), and (\ref{estimate-H-1}), we obtain 
\begin{equation}\label{ineq_dpsie}
\norm{\psi_\e}_{H^1(\Omega)} \le  C \,\e^{1/2} \norm{u_0}_{H^2(\Omega)}.
\end{equation}
Finally, by (\ref{ineq_re}) and (\ref{ineq_dpsie}), 
\begin{equation}\label{estimate-30}
\norm{w_\e}_{H^1(\Omega)} \le \norm{\psi_\e}_{H^1(\Omega)} + \norm{r_\e}_{H^1(\Omega)} 
\le C \,\e^{1/2} \norm{u_0}_{H^2(\Omega)}.
\end{equation}
This completes the proof.
\end{proof}

\begin{rmk}\label{rmk_H1}
{\rm If $D = \partial\Omega$, Theorem \ref{thm_H1} gives the $O(\varep^{1/2})$ error estimate in $H^1$ for the Dirichlet problem.
In the case of the Neumann problem where $D=\emptyset$,
Lemma \ref{lem_BeSedu0} as well as the estimate (\ref{estimate-H-1}) continues to hold.
We now use the second Korn inequality,
\begin{equation}\label{second-K}
\| u\|_{H^1(\Omega)}
\le C \Big\{ \| \nabla u +(\nabla u)^T\|_{L^2(\Omega)} + \sum_{j=1}^m \Big|
\int_\Omega u\cdot \phi_j \Big| \Big\},
\end{equation}
for any $u\in H^1(\Omega; \rd)$,
where $m=d(d+1)/2$, $\big\{ \phi_j: j=1, \dots, m \big\}$ is an orthonormal basis of $\mathcal{R}$, and
$\mathcal{R}=\big\{ u=Cx +D: C^T=-C \in \mathbb{R}^{d\times d}
\text{ and } D\in \rd \big\}$ denotes the space of rigid displacements.
This, together with (\ref{def_elasticity}) and (\ref{estimate-H-1}), gives 
$$
\|\psi_\e\|_{H^1(\Omega)}
\le C \Big\{ \e^{1/2} \| u_0\|_{H^2(\Omega)}
+ \sum_{j=1}^m \Big|
\int_\Omega \psi_\e\cdot \phi_j \Big| \Big\}.
$$
Thus, if  we require that $u_\e, u_0\perp \mathcal{R}$ in $L^2(\Omega; \rd)$,  the estimate (\ref{estimate-30}) still holds.

}
\end{rmk}


\section{Convergence rates in $L^2(\Omega)$}

In this section we give the proof of Theorem \ref{thm_L2}. 
We begin by considering the Neumann boundary value problem
\begin{align}
	\left\{
	\begin{aligned}\label{eq_rhoe_00}
		\cL_\e \rho_\e &= G & \quad &\text{ in } \Omega, \\
		n\cdot A^{\e} \nabla \rho_\e &= h &\quad &\text{ on } \partial\Omega,
	\end{aligned}
	\right.
\end{align}
where $G\in L^2(\Omega; \rd)$, $h\in L^2(\partial\Omega; \rd)$, and
\begin{equation}\label{eq_Neu_cdn}
\int_{\Omega} G + \int_{\partial\Omega} h = 0.
\end{equation}
Recall that a function $\rho_\e\in H^1(\Omega; \rd)$ is called a weak solution of (\ref{eq_rhoe_00}) if
\begin{equation}\label{weak-N}
\int_\Omega A^\e \nabla \rho_\e \cdot \nabla \psi
=\int_\Omega G\cdot \psi +\int_{\partial\Omega} h\cdot \psi
\end{equation}
for any $\psi\in H^1(\Omega; \rd)$.
Under the elasticity condition (\ref{def_elasticity}), it is well known that the Neumann problem (\ref{eq_rhoe_00})
has a unique solution $\rho_\e \in H^1(\Omega; \rd)$
such that $\rho_\e\perp \mathcal{R}$ in $L^2(\Omega; \rd)$.

The homogenized problem for (\ref{eq_rhoe_00}) is given by
\begin{align}
	\left\{
	\begin{aligned}\label{eq_rho0_00}
		\cL_0 \rho_0 & = G  & \quad & \text{ in } \Omega, \\
		n\cdot \widehat{A} \nabla \rho_0  &= h &\quad &  \text{ on } \partial\Omega.
	\end{aligned}
	\right.
\end{align}
If $\Omega$ is $C^{1,1}$,  $G\in L^2(\Omega; \rd)$ and $h\in H^{1/2}(\partial\Omega; \rd)$,
it is known that the unique weak solution of (\ref{eq_rho0_00}) in $H^1(\Omega; \rd)$ with the property
$\rho_0 \perp \mathcal{R}$ in $L^2(\Omega; \rd)$ satisfies
\begin{equation}\label{ineq_rho0_Gg}
\norm{\rho_0}_{H^2(\Omega)} \le C\Big\{ \norm{G}_{L^2(\Omega)} + \norm{h}_{H^{1/2}(\partial\Omega)}\Big\}.
\end{equation}

For the proof of Theorem \ref{thm_L2} we will need to construct a  function $h\in H^{1/2}(\partial\Omega; \rd)$ satisfying (\ref{eq_Neu_cdn}) and
\begin{equation}\label{eq_g_onN}
h = 0 \quad \text{on } N = \partial\Omega\setminus D,
\end{equation}
for each $G\in L^2(\Omega; \rd)$.
This is done in the following lemma.

\begin{lem}\label{lem_G_g}
Let $\Omega$ be a bounded Lipschitz domain and $D$ a closed subset of $\partial\Omega$ with a nonempty interior.
	Let $G\in L^2(\Omega; \rd)$. Then there is a function $h\in H^{1/2}(\partial\Omega; \rd)$ such that 
	$h$ satisfies (\ref{eq_Neu_cdn}), (\ref{eq_g_onN}), and
	\begin{equation}
	\norm{h}_{H^{1/2}(\partial\Omega)} \le C\, \norm{G}_{L^2(\Omega)},
	\end{equation}
	where $C$ depends only on $\Omega$ and $D$.
\end{lem}

\begin{proof}
	By our assumption on $D$ there exist  $x_0\in D$ and $r_0>0$
	 such that $B(x_0,r_0)\cap \partial\Omega \subset D$. 
	 We fix a nonnegative function $h_0 \in C_0^\infty(\R^d)$ satisfying $\text{supp}(h_0) \subset B(x_0,r_0)$ and $h_0 \ge 1$ 
	 in $B(x_0,r_0/2)$. Note that $h_0\in H^1(\partial \Omega)$, $\int_{\partial\Omega} h_0>0$, and $h_0=0$ on $N$.
	Now define
	\begin{equation}
	h = - h_0 \left(\int_{\partial\Omega} h_0    \right)^{-1} \int_{\Omega} G .
	\end{equation}
	Clearly, the function $h$ satisfies (\ref{eq_Neu_cdn}) and (\ref{eq_g_onN}). Moreover,
	\begin{equation}
	\norm{h}_{H^{1/2}(\partial\Omega)} \le \norm{h_0}_{H^{1/2}(\partial\Omega)}
	 \left(\int_{\partial\Omega} h_0    \right)^{-1} |\Omega|^{1/2} \norm{G}_{L^2(\Omega)} = C\norm{G}_{L^2(\Omega)},
	\end{equation}
	where $C$ depends only on $\Omega$ and $D$. 
\end{proof}

Suppose that $\Omega$ is $C^{1,1}$.
By Lemma \ref{lem_G_g} and (\ref{ineq_rho0_Gg}), for each $G\in L^2(\Omega;\rd)$, 
we can construct $h$ so that the weak solution $\rho_0$ of (\ref{eq_rho0_00})
 with the property $\rho_0\perp \mathcal{R}$ in $L^2(\Omega; \rd)$ satisfies
\begin{equation}
\norm{\rho_0}_{H^2(\Omega)} \le C\, \norm{G}_{L^2(\Omega)}.
\end{equation}
Let $\widetilde{\rho}_0 = E\rho_0$  be an extension of $\rho_0$ in $H^2(\R^d;\rd)$ 
and set $\eta_\e = \rho_\e - \rho_0 - \e\chi^{\e} S_\e \nabla \widetilde{\rho}_0$. 
By Remark \ref{rmk_H1}  we see that
\begin{equation}\label{ineq_etae_H1}
\norm{\eta_\e}_{H^1(\Omega)} \le C\, \e^{1/2}\norm{\rho_0}_{H^2(\Omega)} \le C\, \e^{1/2}\norm{G}_{L^2(\Omega)}.
\end{equation}

We are now in a position  to give the proof of Theorem \ref{thm_L2}.

\begin{proof}[\bf Proof of Theorem \ref{thm_L2}]
Let $\psi_\e$, $w_\e$, and $r_\e$ be the same functions as in the proof of Theorem \ref{thm_H1}.
Note that $\psi_\e = w_\e + r_\e = u_\e - u_0 - \e (1-\theta_\e) \chi^\e S_\e \nabla \widetilde{u}_0$. 
Clearly, by Lemma \ref{lem_feSe},
\begin{equation}
\norm{\e (1-\theta_\e) \chi^\e S_\e \nabla \widetilde{u}_0}_{L^2(\Omega)} \le C\, \e \, \norm{u_0}_{H^2(\Omega)}.
\end{equation}
Thus, to prove Theorem \ref{thm_L2},
 it suffices to show $\norm{\psi_\e}_{L^2(\Omega)} \le C\e \norm{u_0}_{H^2(\Omega)}$.
This will be done by a duality argument, using Lemma \ref{lem_BeSedu0}.

 Fix $G\in L^2(\Omega; \rd)$ and let $h\in H^{1/2}(\partial\Omega; \rd)$ be the function given in Lemma \ref{lem_G_g}.
 Let $\rho_\e, \rho_0$ be the weak solutions of (\ref{eq_rhoe_00}) and (\ref{eq_rho0_00}), respectively, 
 such that $\rho_\e, \rho_0 \perp \mathcal{R}$ in $L^2(\Omega; \rd)$.
  Since $\psi_\e \in H^1_D(\Omega;\rd)$ and $n\cdot A^{\e} \nabla \rho_\e  = h=0$ on $N$, by (\ref{weak-N}),
\begin{equation}\label{eq_psieG}
\int_{\Omega} \psi_\e \cdot G = \int_{\Omega} A^\e \nabla \psi_\e \cdot \nabla \rho_\e.
\end{equation}
Write
\begin{equation}\label{eq_J3J4}
\int_{\Omega} A^\e \nabla \psi_\e \cdot \nabla \rho_\e 
= \int_{\Omega} A^\e \nabla w_\e \cdot \nabla \rho_\e + \int_{\Omega} A^\e \nabla r_\e\cdot \nabla \rho_\e = J_3 + J_4.
\end{equation}

We estimate $J_4$ first. Note that,
\begin{align*}
J_4 &= \int_{\Omega} A^\e \nabla r_\e \cdot\nabla \eta_\e + \int_{\Omega} A^\e \nabla r_\e \cdot\nabla \rho_0 + \int_{\Omega} A^\e \nabla r_\e \cdot \nabla (\e\chi^{\e} S_\e \nabla \widetilde{\rho}_0) \\
& = J_{41} + J_{42} + J_{43}.
\end{align*}
In view of (\ref{ineq_re}) and (\ref{ineq_etae_H1}), we obtain
\begin{equation}
|J_{41}| \le C \| \nabla r_\e\|_{L^2\Omega)} \| \nabla \eta_\e\|_{L^2(\Omega)}
 \le C\, \e\,  \norm{u_0}_{H^2(\Omega)} \norm{\rho_0}_{H^2(\Omega)}.
\end{equation}
For $J_{42}$, note that $r_\e$ is supported in $\widetilde{\Omega}_{2\e}$. Hence,
\begin{align*}
|J_{42}| & \le C\, \norm{\nabla r_\e}_{L^2(\Omega)} \norm{\nabla \rho_0}_{L^2(\Omega_{2\e})} \\
& \le C\, \e\,  \norm{u_0}_{H^2(\Omega)} \norm{\rho_0}_{H^2(\Omega)},
\end{align*}
where we have used Lemma \ref{lem_Oeu}  for the last inequality. 
Similarly, 
\begin{equation}
\aligned
|J_{43}|  & 
\le C \, \|\nabla r_\e\|_{L^2(\Omega)} \| \nabla (\e \chi^\e S_\e \nabla \widetilde{\rho}_0)\|_{L^2(\Omega_{2\e})}\\
&\le C\, \e\,  \norm{u_0}_{H^2(\Omega)} \norm{\rho_0}_{H^2(\Omega)},
\endaligned
\end{equation}
where we have used Lemma \ref{lem_OefeSe}.
 As a result, we have proved that
\begin{equation}\label{eq_J4}
|J_4| \le C\, \e\,  \norm{u_0}_{H^2(\Omega)} \norm{\rho_0}_{H^2(\Omega)}.
\end{equation}

It remains to estimate $J_3$. Again, we write
\begin{align*}
J_3 & = \int_{\Omega} A^\e \nabla w_\e \nabla \eta_\e - \int_{\Omega} A^\e \nabla w_\e \nabla \rho_0 - \int_{\Omega} A^\e \nabla w_\e \nabla (\e\chi^{\e} S_\e \nabla \widetilde{\rho}_0) \\
& = J_{31} + J_{32} + J_{33}.
\end{align*}
Note that $J_{31}$ can be easily handled by the $H^1$ estimates of $w_\e$ and $\eta_\e$.
Since the estimate of $J_{32}$ is similar to that of  $J_{33}$, we will  only 
give the estimate for $J_{33}$. To this end, we write
\begin{align}
\begin{aligned}\label{eq_J33}
&\int_{\Omega} A^\e \nabla w_\e \nabla (\e\chi^{\e} S_\e \nabla \widetilde{\rho}_0) \\
&\qquad = \int_{\Omega} A^\e \nabla w_\e \nabla (\theta_{2\e} \e\chi^{\e} S_\e \nabla \widetilde{\rho}_0) 
+ \int_{\Omega} A^\e \nabla w_\e \nabla \big( (1-\theta_{2\e})\e\chi^{\e} S_\e \nabla \widetilde{\rho}_0\big),
\end{aligned}
\end{align}
where $\theta_{2\e}\in C_0^\infty(\rd)$ is a smooth function such that
$\theta_{2\e}(x)=1$ if dist$(x, \partial\Omega)\le 2\e$, 
$\theta_{2\e} (x)=0$ if dist$(x, \partial\Omega)\ge 4\e$, and
$|\nabla \theta_{2\e}|\le C \e^{-1}$.
It follows by Theorem \ref{thm_H1} and Lemma \ref{lem_OefeSe} that
\begin{align}\label{ineq_J33_theta}
\begin{aligned}
\Abs{\int_{\Omega} A^\e \nabla w_\e \nabla (\theta_{2\e} \e\chi^{\e} S_\e \nabla \widetilde{\rho}_0)} 
&\le C\, \e\, \norm{w_\e}_{H^1(\Omega)} \norm{\theta_{2\e} \chi^{\e} S_\e \nabla \widetilde{\rho}_0}_{H^1(\Omega)} \\
& \le C\,\e\, \norm{u_0}_{H^2(\Omega)} \norm{\rho_0}_{H^2(\Omega)}.
\end{aligned}
\end{align}
For the second term in the RHS of (\ref{eq_J33}), 
note that $(1-\theta_\e)\e\chi^{\e} S_\e \nabla \widetilde{\rho}_0 \in H^1_D(\Omega; \rd)$. 
This allows us to apply Lemma \ref{lem_BeSedu0} and obtain
\begin{align}
\begin{aligned}\label{ineq_J33_1-theta}
&\Abs{ \int_{\Omega} A^\e \nabla w_\e \nabla \Big((1-\theta_{2\e})\e\chi^{\e} S_\e \nabla \widetilde{\rho}_0\Big)} \\
&\qquad \le C\, \e\, \norm{u_0}_{H^2(\Omega)} \norm{\nabla 
\big((1-\theta_{2\e})\e\chi^{\e} S_\e \nabla \widetilde{\rho}_0\big)}_{L^2(\Omega)} \\
&\qquad \qquad + C\, \e^{1/2}\norm{u_0}_{H^2(\Omega)} \norm{\nabla \big((1-\theta_{2\e})\e\chi^{\e} S_\e \nabla \widetilde{\rho}_0\big)}_{L^2(\Omega_{2\e})}.
\end{aligned}
\end{align}
Note that the second term vanishes, as $1-\theta_{2\e}$ is supported in $\rd \setminus 
\Omega_{2\e}$. Also,
\begin{equation}
\norm{\nabla \big((1-\theta_{2\e})\e\chi^{\e} S_\e \nabla \widetilde{\rho}_0\big)}_{L^2(\Omega)} \le C\, \norm{\rho_0}_{H^2(\Omega)}.
\end{equation}
This, together with (\ref{ineq_J33_theta}) and (\ref{ineq_J33_1-theta}), leads to
\begin{equation}
|J_{33}| \le C\, \e\, \norm{u_0}_{H^2(\Omega)} \norm{\rho_0}_{H^2(\Omega)}.
\end{equation}
Combining this with the estimates of $J_{31}, J_{32}$, we obtain
\begin{equation}\label{estimate-50}
|J_3| \le C\,\e\, \norm{u_0}_{H^2(\Omega)} \norm{\rho_0}_{H^2(\Omega)}.
\end{equation}
Hence, in view of (\ref{eq_psieG}), (\ref{eq_J3J4}), (\ref{eq_J4}) and (\ref{estimate-50}), we have proved
\begin{equation}\label{estimate-60}
\Abs{\int_{\Omega} \psi_\e \cdot G} \le C\, \e\,  \norm{u_0}_{H^2(\Omega)} \norm{\rho_0}_{H^2(\Omega)} 
\le C\, \e \, \norm{u_0}_{H^2(\Omega)} \norm{G}_{L^2(\Omega)},
\end{equation}
where $C$ depends only on $d$, $\kappa_1$, $\kappa_2$, $D$, and $\Omega$. Therefore, by duality,
\begin{equation}
\norm{\psi_\e}_{L^2(\Omega)} \le C\, \e\,  \norm{u_0}_{H^2(\Omega)},
\end{equation}
which completes  the proof of Theorem \ref{thm_L2}.
\end{proof}

\begin{rmk}
{\rm
If $D=\partial\Omega$, Theorem \ref{thm_L2} gives the sharp $O(\e)$ estimate in $L^2$
for the Dirichlet problem. In the case of the Neumann problem, our proof also gives the estimate
(\ref{optimal}), if we further require that $u_\e, u_0\perp \mathcal{R}$ in  $L^2(\Omega; \rd)$.
To see this, we consider the Neumann problem (\ref{eq_rhoe_00}) with 
$G\in L^2(\Omega; \rd)$, $G\perp \mathcal{R}$, and $h=0$ on $\partial\Omega$.
The same argument as in the proof of Theorem \ref{thm_L2} gives the estimate
(\ref{estimate-60}). By duality this implies that
$$
\|\psi_\e\|_{L^2(\Omega)}
\le C\, \e\, \| u_0\|_{H^2(\Omega)}
+ C \sum_{j=1}^m \Big| \int_\Omega \psi_\e \cdot \phi_j\Big|,
$$
where $m=d(d+1)/2$ and
$\{ \phi_j: j=1, \dots, m \}$ forms an orthonormal basis for $\mathcal{R}$ in $L^2(\Omega; \rd)$.
Using $u_\e, u_0\perp \mathcal{R}$ in $L^2(\Omega; \rd)$, it follows that
$\| \psi_\e\|_{L^2(\Omega)} \le C\, \e \| u_0\|_{L^2(\Omega)}$, 
from which the estimate (\ref{optimal}) follows.
}
\end{rmk}


\section{Interior $H^1$ estimates}
In this section we study the interior $H^1$ convergence and give the proof of Theorem \ref{thm_H1_int}.

\begin{lem}\label{lem_wgt_est}
Let $w_\e$ be defined by (\ref{def_we}).
	Let $\zeta\in W^{1,\infty}(\Omega)$ be a nonnegative function in $\Omega$ 
	such that $\zeta=0$ on $\partial\Omega$. Then,
	\begin{align*}
	\| \zeta \nabla w_\e\|_{L^2(\Omega)}
	\le C \| u_0\|_{H^2(\Omega)}
	\Big\{ \e\, \| \zeta \|_{W^{1, \infty}(\Omega)} 
	+\e^{1/2} \| \zeta \|_{L^\infty(\Omega_{2\e})}
	+\e^{3/4} \|\zeta \nabla \zeta\|_{L^\infty(\Omega_{2\e})}^{1/2} \Big\},
	\end{align*}
	where $C$ depends only on $d$, $\kappa_1$, $\kappa_2$, $D$, and $\Omega$.
\end{lem}

\begin{proof} Since $\zeta w_\e\in H^1_0(\Omega; \rd)$, it follows from the elasticity condition and
the first Korn inequality that 
\begin{equation}\label{estimate-70}
\aligned
\|\zeta \nabla w_\e\|^2_{L^2(\Omega)}
&\le 2\|\nabla (\zeta w_\e)\|^2_{L^2(\Omega)} +2\| w_\e \nabla \zeta \|^2_{L^2(\Omega)}\\
& \le C \int_\Omega A^\e \nabla (\zeta w_\e)\cdot \nabla (\zeta w_\e) 
+2\| w_\e\|^2_{L^2(\Omega)} \|\nabla \zeta\|^2_{L^\infty(\Omega)}\\
&\le C \int_\Omega A^\e \nabla  w_\e \cdot \nabla (\zeta ^2w_\e) 
+C \| w_\e\|^2_{L^2(\Omega)} \|\nabla \zeta\|^2_{L^\infty(\Omega)},
\endaligned
\end{equation}
where we also used the identity
$$
 A^\e \nabla (\zeta w_\e) \cdot \nabla (\zeta w_\e) 
	=  A^\e \nabla w_\e  \cdot \nabla (\zeta^2 w_\e) 
	+A^\e  (w_\e\nabla \zeta) \cdot (w_\e \nabla \zeta) .
$$
Note that by Lemma \ref{lem_BeSedu0},
	\begin{align}\label{ineq_wez0}
	\begin{aligned}
	&{\int_{\Omega} A^\e \nabla w_\e \cdot \nabla (\zeta^2 w_\e )} \\
	&\qquad \le C\,\e\, \norm{u_0}_{H^2(\Omega)} \norm{\nabla (\zeta^2 w_\e )}_{L^2(\Omega)} + C\, \e^{1/2}\norm{u_0}_{H^2(\Omega)} \norm{\nabla (\zeta^2 w_\e )}_{L^2(\Omega_{2\e})}.
	\end{aligned}
	\end{align}
	This, together with (\ref{estimate-70}), gives
	\begin{equation}
	\aligned
	\|\zeta \nabla w_\e \|^2_{L^2(\Omega)}
	 \le C\,\e\,  & \| u_0\|_{H^2(\Omega)} \|\zeta \nabla w_\e\|_{L^2(\Omega)} \|\zeta\|_{L^\infty(\Omega)}\\
&	+   C\,\e \,\| u_0\|_{H^2(\Omega)} \| w_\e\|_{L^2(\Omega)} \|\zeta\nabla \zeta\|_{L^\infty(\Omega)}\\
& 	+ C\, \e^{1/2} \| u_0\|_{H^2(\Omega)} \| \zeta \nabla w_\e\|_{L^2(\Omega_{2\e})} \| \zeta\|_{L^\infty(\Omega_{2\e})}\\
&	+C\, \e^{1/2} \| u_0\|_{H^2(\Omega)} \| w_\e\|_{L^2(\Omega)} \| \zeta \nabla \zeta\|_{L^\infty(\Omega_{2\e})}\\
&	+ C\, \| w_\e\|^2_{L^2(\Omega)} \|\nabla \zeta\|^2_{L^\infty(\Omega)}.
\endaligned
	\end{equation}
	By the Cauchy inequality with an $\e>0$ we obtain 
	\begin{equation}\label{estimate-100}
	\aligned
\|\zeta \nabla w_\e \|^2_{L^2(\Omega)}
	\le C\,\e^2   & \| u_0\|^2_{H^2(\Omega)}  \|\zeta\|_{L^\infty(\Omega)}^2
	  + C\,\e \,\| u_0\|_{H^2(\Omega)} \| w_\e\|_{L^2(\Omega)} \|\zeta\nabla \zeta\|_{L^\infty(\Omega)}\\
& 	+ C \,\e \,\| u_0\|^2_{H^2(\Omega)}  \| \zeta\|^2_{L^\infty(\Omega_{2\e})}\\
&	+C \e^{1/2} \| u_0\|_{H^2(\Omega)} \| w_\e\|_{L^2(\Omega)} \| \zeta \nabla \zeta\|_{L^\infty(\Omega_{2\e})}\\
&	+ C \, \| w_\e\|^2_{L^2(\Omega)} \|\nabla \zeta\|^2_{L^\infty(\Omega)}.
\endaligned
	\end{equation}
	It then follows by the estimate $\| w_\e\|_{L^2(\Omega)} \le C \varep \| u_0\|_{H^2(\Omega)}$ that
	$$
	\|\zeta \nabla w_\e \|^2_{L^2(\Omega)}
	\le C \, \| u_0\|^2_{H^2(\Omega)}\Big\{ \e^2 \|\zeta\|^2_{W^{1, \infty}(\Omega)}
+\e\, \|\zeta\|^2_{L^\infty(\Omega_{2\e})}
+\e^{3/2} \|\zeta\nabla \zeta\|_{L^\infty(\Omega_{2\e})} \Big\}.
$$
	This completes the proof.
\end{proof}

\begin{proof}[\bf Proof of Theorem \ref{thm_H1_int}]

Let $\zeta (x) =\delta(x) =\text{dist}(x, \partial\Omega)$.
Note that $\zeta=0$ on $\partial\Omega$ and $\| \zeta \|_{W^{1, \infty}(\Omega)} \le C$, where $C$
depends only on $\Omega$.
Theorem \ref{thm_H1_int} now follows readily from Lemma \ref{lem_wgt_est}.
\end{proof}
As a corollary, we obtain the following interior estimate.

\begin{cor}\label{cor-10}
Let $\Omega^\prime$ be an open subset of $\Omega$ such that $\text{dist}(\Omega^\prime, \partial \Omega) > 0$.
	Under the same conditions as in Theorem \ref{thm_L2}, we have
	\begin{equation}
		\norm{u_\e -  u_0 - \e \chi^\e S_\e \nabla \widetilde{u}_0}_{H^1(\Omega')} \le C\, \e\, \norm{u_0}_{H^2(\Omega)},
	\end{equation}
	where $C$ depends only on $d$, $\kappa_1$, $\kappa_2$, $D$, $\Omega^\prime$ and $\Omega$.
\end{cor}

\begin{rmk}
{\rm The estimates in Lemma \ref{lem_wgt_est} and Theorem \ref{thm_H1_int} as well as 
in Corollary \ref{cor-10} continue to hold for the Neumann boundary value problems, if we further require $u_\e, u_0\perp \mathcal{R}$
in $L^2(\Omega; \rd)$.
The proof is exactly the same.
}
\end{rmk}

\medskip

\noindent{\bf Acknowledgement.} Both authors are supported by NSF grant DMS-1161154.

\bibliographystyle{amsplain}
\bibliography{Shen-Zhuge-1.bbl}
\end{document}